\documentclass{amsart}
\usepackage{amssymb} 
\usepackage{soul,xcolor}
\soulregister\cite7
\soulregister\ref7
\soulregister\pageref7
\setstcolor{red}
\setul{}{1.5pt}

\newtheorem{thm}{Theorem}[section]
\newtheorem{lem}[thm]{Lemma}
\newtheorem{lem-dfn}[thm]{Lemma-Definition}
\newtheorem{prop}[thm]{Proposition}
\newtheorem{cor}[thm]{Corollary}

\newtheorem{ass}[thm]{Assumption}

\newtheorem{introthm}{Theorem}
\theoremstyle{definition}
\newtheorem{defn}[thm]{Definition}

\newtheorem{ex}[thm]{Example}

\newtheorem{prob}[thm]{Problem}

\theoremstyle{remark}

\newtheorem{rem}[thm]{Remark}
\numberwithin{equation}{section}

\newcommand{\thmref}[1]{Theorem~\ref{#1}}
\newcommand{\lemref}[1]{Lemma~{\ref{#1}}}
\newcommand{\corref}[1]{Corollary~\ref{#1}}
\newcommand{\proref}[1]{Proposition~\ref{#1}}
\newcommand{\remref}[1]{Remark~\ref{#1}}

\newcommand{\defref}[1]{Definition~\ref{#1}}
\newcommand{\assref}[1]{Assumption~\ref{#1}}
\newcommand{\exref}[1]{Example~\ref{#1}}

\newcommand{\sref}[1]{Section~\ref{#1}}

\DeclareMathOperator{\Spec}{Spec}
\DeclareMathOperator{\spec}{Spec}

\DeclareMathOperator{\proj}{Proj}

\DeclareMathOperator{\supp}{Supp}

\DeclareMathOperator{\Img}{Im}

\DeclareMathOperator{\emb}{embdim}

\DeclareMathOperator{\di}{div}

\DeclareMathOperator{\sing}{Sing} 
 
\DeclareMathOperator{\mult}{mult}

\newcommand{\m}{\mathfrak m}

\newcommand{\q}{\mathfrak q}

\newcommand{\Z}{\mathbb Z}
\newcommand{\Q}{\mathbb Q}

\newcommand{\C}{\mathbb C}

\newcommand{\cC}{\mathcal C}

\newcommand{\cL}{\mathcal L}

\newcommand{\cO}{\mathcal O}

\renewcommand{\:}{\colon}

\newcommand{\ol}[1]{\overline {#1}}

\newcommand{\defset}[2]{{\left\{#1\,\left| \,#2 \right. \right\}}}

\newcommand{\zb}[1]{Z_{B_{#1}}}

\title[Cohomology of ideals in surface singularities]{Cohomology of ideals in elliptic surface singularities}

\author{Tomohiro Okuma}
\address{Department of Mathematical Sciences,
Faculty of Science, Yamagata University, Yamagata, 990-8560, Japan.}
\email{okuma@sci.kj.yamagata-u.ac.jp}



\thanks{This work was partially supported by JSPS  KAKENHI 
Grant Number 26400064.}

\begin{document}

\begin{abstract}
We introduce the the normal reduction number of two-dimensional normal singularities and prove that 
elliptic singularity has normal reduction number two.
We also prove that for a  two-dimensional normal singularity which is not rational, 
it is Gorenstein and its maximal ideal is a $p_g$-ideal if and only if it is a maximally elliptic singularity of degree $1$.
\end{abstract}

\maketitle


\section{Introduction}
Let $(A,\m)$ be an excellent two-dimensional normal local domain containing an algebraically closed field isomorphic to the residue field. 
In this paper, we simply call such a local ring a {\em normal surface singularity}.
 Lipman \cite{Li} proved that if $(A,\m)$ is a rational singularity, then for any integrally closed $\m$-primary ideals $I$ and $I'$ 
we have that the product $II'$ is also integrally closed and that 
  $I^2=QI$ for any minimal reduction $Q$ of $I$.
Cutkosky \cite{CharRat} showed that the first property characterizes 
the two-dimensional rational singularities.
In \cite{OWYgood}, \cite{OWYrees}, \cite{OWYcore}, we introduced the notion of $p_g$-ideals, which satisfy the properties above, and proved many nice properties. For any normal surface singularity, $p_g$-ideals exist plentifully and form a semigroup with respect to the product.
It is easy to see that $A$ is a rational singularity if and only if every integrally closed $\m$-primary ideal is a 
$p_g$-ideal (see \remref{r:rat}).
So it is natural to ask how the semigroup of the $p_g$-ideals encodes the properties of the singularity.

Let $X\to \spec A$ be a resolution of singularity. 
Suppose that an integrally closed $\m$-primary ideal $I$ is represented by a cycle $Z$ on $X$ (see \S \ref{ss:pg}). Then $I=H^0(X, \cO_X(-Z))$. We define an invariant $q(I)$ to be $\ell_A(H^1(X, \cO_X(-Z)))$, where $\ell_A$ denotes the length of $A$-modules.
Then $I$ is called the {\em $p_g$-ideal} if $q(I)=p_g(A)$, where $p_g$ denotes the geometric genus (see \defref{d:q}). 
In general, we have $p_g(A)\ge q(\ol{{I^n}})\ge q(\ol{{I^{n+1}}})$ 
(see \proref{p:le-cycle}), where $\ol{I^n}$ denotes the integral closure of $I^n$, and we know that there exist ideals with $q=0$ and $q=p_g(A)$; however, the range of $q$ is still unknown.
 We are interested in obtaining the range of $q$ and also the minimal integer $n_0$ such that $q(\ol{I^n})=q(\ol{I^{n_0}})$ for $n\ge n_0$. This integer connects with the {\em normal reduction number} $\bar{r}(I)$ (see \S \ref{s:nrn}).
The results of Lipman and Cutkosky above implies that $\bar r(A)=1$ if and only if $A$ is a rational singularity (\thmref{t:r1}).
Then a very simple question arises: can we characterize normal surface singularities with  $\bar r(A)=2$? 

In this paper, we give partial answers to the questions above. We will prove the following (see \thmref{t:ellr}, \corref{c:P}, \thmref{t:mpg}).

\begin{introthm}\label{t:int}
{\rm (1)}
If $A$ is an elliptic singularity, then $\bar r(A)=2$, and for any $0\le q \le p_g(A)$ there exists an integrally closed $\m$-primary ideal $I$ with $q(I)=q$.


{\rm (2)} Assume that $A$ is not rational.
Then $A$ is Gorenstein and $\m$ is a $p_g$-ideal if and only if $A$ is a maximally elliptic singularity with $-Z_E^2=1$, where $Z_E$ is the fundamental cycle on a resolution.
\end{introthm}


Throughout this paper, we assume the following.

\begin{ass}\label{ass:NC}
For any integrally closed $\m$-primary ideal $I\subset A$ represented on a resolution $X\to \spec A$ with exceptional set $E$, and for a general element $h\in I$, if $H$ denotes the the strict transform of $\di_{\spec A}(h)$ on $X$, then $H$ is a reduced divisor which is a disjoint union of nonsingular curves and each component of $H$ intersects the exceptional set transversally, namely, the local equations of $H$ and $E$ generate the maximal ideal at the intersection point.
(This condition holds in case the singularity is defined over a field of characteristic zero.)
 \end{ass} 

This paper is organized as follows. In \sref{s:Pre}, we recall the definitions and several properties of elliptic singularities and $p_g$-ideals in normal surface singularities which are needed later. In \sref{s:nrn}, we introduce  the normal reduction number and study the invariant $q$, and then prove (1) of \thmref{t:int}. In the last section, we prove  (2) of \thmref{t:int} and give an example of non-Gorenstein elliptic singularity  with $-Z_E^2=1$ of which the maximal ideal is a $p_g$-ideal.

\subsection*{Acknowledgements}
 The author would like to thank Professor Kei-ichi Watanabe and Professor Ken-ichi Yoshida for inviting him to those research and  valuable discussions. He was told the notion of the normal reduction number by Professor Watanabe. 
The author would also like to thank Professor Masataka Tomari for letting us know the result of \thmref{t:mpg}.

The author is grateful to the referee for very careful reading of the paper, and for valuable suggestions and comments which are very helpful for improving the paper.

\section{Preliminaries}
\label{s:Pre}
Throughout this paper, let $(A,\m)$ denote a normal surface singularity, namely, an excellent 
two-dimensional normal local domain containing an algebraically closed field isomorphic to the residue field
and 
$f\:X \to \spec A$ a resolution of singularity
with exceptional set $E:=f^{-1}(\m)$. 
Let $E=\bigcup_{i=1}^rE_i$ be the decomposition into irreducible components of $E$.
A divisor on $X$ supported in $E$ is called a {\it cycle}.
A divisor $D$ on $X$ is said to be {\em nef} if $DE_i\ge 0$ for all $E_i\subset E$, 
where $DE_i$ denotes the intersection number.
A divisor $D$ is said to be {\em anti-nef} if $-D$ is nef.
Since the intersection matrix is negative definite, there exists an anti-nef cycle $Z\ne 0$ and it satisfies  $Z\ge E$.

For a cycle $B > 0$, 
we denote by $\chi(B)$ the Euler characteristic $\chi(\cO_B)$. 
We have $\chi(D)+\chi(F)-D  F=\chi(D+F)$.
By definition, $p_a(B)=1-\chi(B)$.
The {\em fundamental cycle} on  $\supp (B)$ is denoted by $Z_B$; by definition, $Z_B$ is the minimal cycle such that $\supp (Z_B)=\supp (B)$ and $Z_BE_i\le 0$ for all $E_i \le B$.

For any function $h \in H^0(\cO_X)\setminus\{0\}$, which has zero of order $a_i$ at
$E_i$, we put $(h)_E=\sum a_iE_i$.
Clearly the cycle $(h)_E$ is anti-nef.

\subsection{Elliptic singularities}

\begin{defn}[Wagreich {\cite[p. 428]{wag.ell}}]
A normal surface singularity $(A,\m)$ is called an {\em elliptic singularity} if one of the following equivalent conditions  holds:
\begin{enumerate}
 \item $\chi(D)\ge 0$ for all cycles $D>0$ and $\chi(F)=0$ for some
       cycle $F>0$;
 \item  $\chi(Z_E)=0$.
\end{enumerate}
\end{defn}


\begin{rem}
The proof of the implication (2) $\Rightarrow$ (1) is given by several authors: e.g., Laufer \cite[Corollary 4.2]{la.me},
Tomari  \cite[Theorem (6.4)]{tomari.ell}.
See also \cite[Remark (6.5)]{tomari.ell}.
\end{rem}


\begin{defn}[Laufer {\cite[Definition 3.1 and 3.2]{la.me}}]
Suppose that $(A,\m)$ is an elliptic singularity.
Then there exists a unique cycle $E_{min}$ such that $\chi(E_{min})=0$ and
 $\chi(D)>0$ for all cycles $D$ such that $0<D<E_{min}$.
The cycle $E_{min}$ is called a {\em minimally elliptic cycle}.
The singularity $(A,\m)$ is said to be minimally elliptic if the fundamental
 cycle is  minimally elliptic on the minimal resolution.
\end{defn}

The next proposition follows from \cite[Proposition 3.2]{la.me}.

\begin{prop}\label{p:Eexists}
Assume that $A$ is an elliptic singularity. 
Let $D>0$ be a cycle with $\chi(D)=0$.
Then we have the following.
\begin{enumerate}
 \item $D\ge E_{min}$. Consequently, $D$ is connected (i.e., $\supp (D)$ is connected).
 \item Any connected reduced cycle $F$ 
not containing any component of $D$
       is the  exceptional set of a rational singularity and satisfies $D  F\le 1$.  
\end{enumerate} 
\end{prop}

The notion of elliptic sequence was introduced by S.~S.-T.~Yau \cite{yau.gor},
\cite{yau.hyper} for elliptic singularities. 

\begin{defn}
Assume that $(A,\m)$ is an elliptic singularity.
Let $B$ be a connected reduced cycle such that $\supp (E_{min})\subset B$.
We define the {\em elliptic sequence} on $B$ as follows:
Let $B_0=B$. If $Z_{B_0}   E_{min}<0$, then the elliptic
 sequence is $\{Z_{B_0}\}$.
If $Z_{B_i}   E_{min}=0$, then define $B_{i+1}\le B_i$ to be the
 maximal reduced connected cycle containing $\supp(E_{min})$ such that
 $Z_{B_{i}}  B_{i+1}=0$.
If we have $Z_{B_m}   E_{min}<0$, then the elliptic
 sequence is $\{Z_{B_0}, \dots ,Z_{B_m}\}$.
\end{defn}

\begin{prop}[{Tomari
\cite[Theorem (6.4)]{tomari.ell}}]\label{p:tomari-nef}
Let $\{Z_{B_0}, \dots ,\zb m\}$ be the elliptic sequence on $B$.
For an integer $0 \le t \le m$, we define a cycle $C_t$ by
$$
C_t=\sum _{i=0}^tZ_{B_i}.
$$
Then the set $\defset{C_k}{0 \le k \le m}$  coincides with the set of cycles $C>0$ supported on $B$ such that $C$ is anti-nef on $B$ and $\chi(C)=0$.
\end{prop}

\begin{lem}[R{\"o}hr {\cite[1.7]{rohr}}, cf. {\cite[Lemma 3.2]{o.numGell}}]\label{l:rohr}
Assume that $A$ is an elliptic singularity. 
Let $D$ be a nef divisor on $X$ such that $D  E_{min}>0$.
Then $H^1(\cO_X(D))=0$.
\end{lem}

\subsection{$p_g$-ideals}\label{ss:pg}
Let $I\subset A$ be an integrally closed $\m$-primary ideal.
Then there exists a resolution $X \to \Spec A$ and a cycle $Z>0$ on $X$ 
such that $I \cO_X = \cO_X(-Z)$.
In this case, we denote the ideal $I$ by $I_Z$,
and we say that $I$ is  {\em represented on $X$ by $Z$}.
Note that  $I_Z=H^0(X,\cO_X(-Z))$.

When we write $I_Z$, we always assume that $\cO_X(-Z)$ is generated by global sections, namely, $I \cO_X = \cO_X(-Z)$.

We denote by $h^1(\cO_X(-Z))$ the length $\ell_A(H^1(X,\cO_X(-Z)))$. 

\begin{defn}\label{d:q}
The {\em geometric genus} $p_g(A)$ of $A$ is defined by 
$p_g(A) = h^1(\cO_X)$.
We define an invariant $q(I)$ by $q(I)=h^1(\cO_X(-Z))$; this does not depend on the choice of representations of the ideal (see \cite[Lemma 3.4]{OWYgood}).
\end{defn}


Kato's Riemann-Roch formula \cite{kato} shows a relation between the colength $\ell_A(A/I)$ and the invariant  $q(I)$ of $I=I_Z$:
\[
\ell_A(A/I) + q(I)=-\dfrac{Z^2+K_XZ}{2}+p_g(A).
\]
In particular, $\ell_A(A/I)$ can be computed from the resolution graph if $I$ is a $p_g$-ideal (see \defref{p_g-dfn}). 
However, the computation of the invariant  $q(I)$ (or $\ell_A(A/I)$) is very difficult for non rational singularities, and it seems to be given only for very special cases (e.g., \cite[\S 7]{OWYgood}).

\smallskip

We say that $\cO_X(-Z)$ {\em has no fixed component}
if $H^0(\cO_X(-Z))\ne H^0(\cO_X(-Z-E_i))$ for every $E_i\subset E$;
this is equivalent to the existence of an element $h\in H^0(\cO_X(-Z))$ such that $(h)_E=Z$.
It is clear that $\cO_X(-Z)$ has no fixed component when $I$ is represented by $Z$.

\begin{prop}[{\cite[2.5, 3.1]{OWYgood}}]\label{p:le-cycle}
Let $Z'$ and $Z$ be  cycles on $X$ and assume that $\cO_X(-Z)$ has no fixed components.
Then we have 
\[
h^1(\cO_X(-Z'-Z)) \le  h^1(\cO_X(-Z')).
\]  
 In particular, $h^1(\cO_X(-Z))\le p_g(A)$; if the equality holds, then $\cO_X(-Z)$ is generated by global sections.
\end{prop}

\begin{defn}\label{p_g-dfn}
(1) We call $I$ a {\em $p_g$-ideal} if $q(I)=p_g(A)$.  

(2) A cycle $Z>0$ is called a {\em $p_g$-cycle} if $\cO_X(-Z)$ is generated by global sections and $h^1(\cO_X(-Z))=p_g(A)$.
\end{defn}

\begin{rem}\label{r:rat}
If $A$ is {\em rational}, namely $p_g(A)=0$, every integrally closed $\m$-primary ideal is a 
$p_g$-ideal by \cite[12.1]{Li}. Conversely, this property characterizes  a rational singularity
because we always have integrally closed $\m$-primary ideal $I$ with $q(I)=0$ (see e.g. \cite[4.5]{OWYgood}).
\end{rem}

In \cite{OWYgood} and \cite{OWYrees}, we obtained many good properties and characterizations of $p_g$-ideals.
Let us review some of these results.

Recall that an ideal $J \subset I$ is called a {\it reduction} of $I$ if $I$ is integral over $J$ or, 
equivalently,  $I^{r+1} = I^rJ$ for some integer $r\ge 1$ (see e.g. \cite{HS-book}).
An ideal  $Q\subset I$ is called a 
{\it minimal reduction} of $I$ if $Q$ is minimal among the reductions of $I$. In our case, any minimal reductions of an $\m$-primary ideal is a parameter ideal (cf. \cite[8.3]{HS-book}).

\begin{prop}[\textrm{see \cite[3.6]{OWYgood}}]\label{p:sg}
Let  $I$ and $I'$  be any integrally closed $\m$-primary ideals of $A$.
Then we have the following.
\begin{enumerate}
\item $I$ and $I'$  are $p_g$-ideals if and only if so is $II'$. 
In particular, the set of $p_g$-ideals forms a semi group with respect to the product.
\item If $I$ is a $p_g$-ideal and $Q$ a minimal reduction of $I$, then $I^2=QI$.
\end{enumerate}
\end{prop}

Next we recall a characterization of $p_g$-ideals by cohomological cycle.  
Let $K_X$ denote the canonical divisor on $X$.
Let $Z_{K_X}$ denote the {\em canonical cycle}, 
i.e., the $\Q$-divisor supported in $E$ 
such that $(K_X+Z_{K_X}) E_i=0$ for every $E_i\subset E$.
By \cite[\S 4.8]{chap}, if $p_g(A)>0$, 
there exists the smallest cycle $C_X>0$ on $X$ such that $h^1(\cO_{C_X}) =p_g(A)$; if $A$ is Gorenstein and the resolution $f \colon X \to \Spec A$ is minimal, then $C_X=Z_{K_X}$.
The cycle $C_X$ is called the {\em cohomological cycle} on $X$.
We put $C_X=0$ if $A$ is a rational singularity.

\begin{prop}[cf. {\cite[Proposition 2.6]{OWYcore}}]\label{p:CC}
Let $C\ge 0$ be the minimal cycle such that 
$H^0(X\setminus E, \cO_X(K_X))=H^0(X,\cO_X(K_X+C))$.
Then $C$ is the cohomological cycle.
Therefore 
if $g\:X'\to X$ is the blowing-up at a point in $\supp (C_X)$ and $E_0$ the exceptional set of $g$, then $C_{X'}=g^*C_X-E_0$.
For any cycle $D>0$ without common components with $C_X$, we have $h^1(\cO_D)=0$. 
\end{prop}

\begin{prop} [{\cite[3.10]{OWYgood}}] \label{p:CX} 
Assume that $p_g(A)>0$.
Let $Z>0$ be a cycle such that $\cO_X(-Z)$ has no fixed component.
Then $Z$ is a $p_g$-cycle if and only if
$\cO_{C_X}(-Z)\cong \cO_{C_X}$.
\end{prop}

\begin{prop}[{\cite{OWYrees}}]\label{p:pgC}
Let $I$ be an integrally closed $\m$-primary ideal.
Then $I$ is a $p_g$-ideal if and only if
the Rees algebra $\bigoplus_{n\ge 0}I^nt^n\subset A[t]$ is a Cohen-Macaulay normal domain. 
\end{prop}

The following theorem shows that the $p_g$-ideals exist plentifully.

\begin{thm}[cf. {\cite[Theorem 5.1]{OWYcore}}]
\label{t:pgf}
Let $I$ be an integrally closed $\m$-primary ideal and $g$ an arbitrary element of $I$. 
Then there exists $h\in I$ such that the integral closure of the ideal $(g,h)$ is a $p_g$-ideal. 
\end{thm}

\section{The normal reduction number}
\label{s:nrn}


\begin{defn}
Let $I$ be an integrally closed $\m$-primary ideal and $Q$ a minimal reduction of $I$.
We define the {\em normal reduction number} $\bar{r}$ of $I$ by
\[
\bar{r}( I ) = \min \defset{ r \in \Z_{\ge 0} }{\ol{I^{n+1}} = Q \ol{ I^n} \text{ for all $n \ge r$} }.
\]
We shall see that $\bar r(I)$ is independent of the choice of minimal reductions by \corref{c:r=n0+1}.
Let 
\[
\bar r (A)=\max \defset{\bar{r}( I ) }{\text{ $I$ is an integrally closed $\m$-primary ideal of $A$}}.
\]
\end{defn}


The normal reduction number has been studied by many authors implicitly or explicitly in the context of the Hilbert function and the Hilbert polynomial associated with $\{\ol{I^{n}}\}_{n\ge 0}$ (e.g., \cite{mo.Cl}, \cite{ito.Int}, \cite{Hun.Hilb}).
We study this invariant in terms of cohomology of ideal sheaves of cycles toward a geometric understanding of the normal reduction number.

If $A$ is rational, then by Lipman \cite{Li} (cf. \proref{p:sg}), we have $\ol{{I^2}}=I^2=QI$ for any integrally closed $\m$-primary ideal $I$.
On the other hand, Cutkosky \cite{CharRat} proved that the converse holds too.
Hence we have the following.

\begin{thm}\label{t:r1}
$\bar r(A)=1$ if and only if $A$ is a rational singularity.
\end{thm}
Note that the rationality is determined by the resolution graph (see \cite{artin.rat}).

The main result of this section is the following.

\begin{thm}\label{t:ellr}
If $A$ is an elliptic singularity, then $\bar r(A)=2$.
\end{thm}

\begin{defn}
Let $D\ge 0$ be an effective cycle and let
\[
h(D)=\max  \defset{h^1(\cO_{D'})}{D'\ge 0, \supp (D')\subset \supp (D)},
\]
where we put $h^1(\cO_{D'})=0$ if $D'=0$.
There exists a unique minimal cycle $C$ such that $h^1(\cO_C)=h(D)$
(cf. \cite[\S 4.8]{chap}). 
We call $C$ the {\em cohomological cycle on $D$}.
We define a reduced cycle $D^{\bot}$ to be the sum of the components $E_i\subset E$ such that $DE_i=0$.
\end{defn}


\begin{rem}\label{r:triv-bot}
Suppose that $\cO_X(-Z)$ has no fixed component.
Then there exists  a function $h\in H^0(\cO_X(-Z))$ such that $\di_X(h)=Z+H$, where $H$ is the strict transform of $\di_{\spec A}(h)$. Since $ZE_i=-HE_i$ for any $E_i\subset E$, it follows that $\supp (Z^{\bot})$ and $\supp (H)$ have no intersection.
Thus for any cycle $F>0$ supported in $Z^{\bot}$, we have $\cO_F(-Z)=\cO_F(-\di_X(h))\cong \cO_F$.  
\end{rem}

Let $Z>0$ be a cycle on $X$ and let $\cL(n)=\cO_X(-nZ)$.

If $\cO_X(-Z)$ has no fixed component, we define an integer $n_0(Z)$ by 
\[
n_0(Z)=\min\defset{n\in \Z_{\ge 0}}{h^1(\cL(n))= h^1(\cL(m)) \text{ for } m\ge n}.
\] 
This is well-defined by \lemref{l:nZ} (1).

\begin{lem}[See {\cite[3.1 and 3.4]{OWYrees}}]\label{l:nZ}
Suppose that $\cO_X(-Z)$ has no fixed component.
Let $C$ denote the cohomological cycle on $Z^{\bot}$. 
Then we have the following.
\begin{enumerate}
\item $h^1(\cL(n))\ge h^1(\cL(n+1))$ for $n\ge 0$.
\item If $\cO_X(-Z)$ is generated by global sections, then
\[
n_0(Z)=\min\defset{n\in \Z_{\ge 0}}{h^1(\cL(n))= h^1(\cL(n+1))}.
\] 
If $Z$ is a $p_g$-cycle, then $n_0(Z)=0$.

\item Let $n_0=n_0(Z)$.
Then $\cO_C(-n_0Z)\cong \cO_C$ and $h^1(\cL(n_0(Z)))=h^1(\cO_C)$. 
\item $\cL(n)$ is generated by global sections for $n>n_0$.
\end{enumerate}
\end{lem}
\begin{proof}
The claims (1)--(3) are proved in \cite{OWYrees}. 
Let $h\in I_Z$ be a general element and consider the exact sequence
\[
0 \to \cL((n-1))\xrightarrow{\times h} \cL(n) \to \cC(n) \to 0,
\]
where $\cC(n)$ is supported on the divisor $\di_X(h)-(h)_E$.
If $n> n_0(Z)$, then $H^0(\cL(n)) \to H^0(\cC(n))$ is surjective since $H^1(\cC(n))=0$.
This shows that $H^0(\cL(n))$ has no base points.
\end{proof}

\begin{defn}
For an integrally closed $\m$-primary ideal $I$ represented by $Z$, let $n_0(I)=n_0(Z)$; this is independent of the choice of representations since so is $q(I)$.
\end{defn}


\begin{rem}
Let us explain the invariant $q(I_{n_0Z})$ in terms of ``partial resolution.''
Suppose that $I$ is represented by a cycle $Z>0$ on $X$.
Let $Y$ be the normalization of the blowing-up of $\spec A$ by $I$, namely, $Y=\proj \bigoplus_{n\ge 0}I_{nZ}t^n$.
Let $\phi \: X\to Y$ be the natural morphism and let $Z'=\phi_*Z$. Then $I\cO_Y=\cO_Y(-Z')$.
Since $\phi_*\cO_X=\cO_Y$, from Leray's spectral sequence,
we obtain the following exact sequence for $n\ge 0$.
\begin{equation}
\label{eq:Y}
0 \to H^1(\cO_Y(-nZ')) \to H^1(\cO_X(-nZ)) \to H^0(R^1\phi_*\cO_X\otimes \cO_Y(-nZ')) \to 0.
\end{equation}
Let $\sing (Y)$ denote the set of singular points of $Y$.
Since the support of $R^1\phi_*\cO_X\otimes \cO_Y(-nZ')$ is contained in $\sing(Y)$, we obtain that $R^1\phi_*\cO_X\otimes \cO_Y(-nZ')\cong R^1\phi_*\cO_X$.
It follows from \lemref{l:nZ} (3) that 
\[
\ell_A(R^1\phi_*\cO_X)=\sum_{y\in \sing(Y)}p_g(Y,y)=q(I_{n_0Z}).
\]
The sequence \eqref{eq:Y} implies the following equalities.
\begin{gather*}
q(I_{n_0Z})=p_g(A)-h^1(\cO_Y)=h^1(\cO_X(-nZ)) \; 
\text{ for $n\ge n_0(I)$,} \\
q(I_{nZ})-q(I_{n_0})=h^1( \cO_Y(-nZ')).
\end{gather*}
In particular, $h^1( \cO_Y(-nZ'))=0$ if and only if $n\ge n_0$.
\end{rem}

\begin{cor}\label{c:r=n0+1}
Let $I$ be an integrally closed $\m$-primary ideal represented by $Z$. Then
$\bar r (I)=n_0(I)+1$.
\end{cor}
\begin{proof}
Let $Q=(f_1, f_2)\subset I_Z$ a minimal reduction of $I_Z$. Then for any integer $n$, we have the following exact sequence.
\begin{equation} \label{Secondeq}
0 \to \cL(n-1)
\stackrel{(f_1,f_2)}{\longrightarrow} \cL(n)^{\oplus 2}
\stackrel{\genfrac{(}{)}{0pt}{}{-f_2}{f_1}}{\longrightarrow}
\cL(n+1) \to 0.
\end{equation}
From \lemref{l:nZ} (1), (2) and the sequence \eqref{Secondeq}, for an arbitrary integer $r\ge 0$,
we have that 
$QI_{nZ}=I_{(n+1)Z}$ for all $n\ge r$ if and only if $h^1(\cL(n))=h^1(\cL(r-1))$ for all $n\ge r$.
\end{proof}


\begin{rem}
In \cite[Corollary 14]{ito.Int}, Ito proved that if $p_g(A)=1$, then $\m^3=\q\m^2$, where $\q$ is a minimal reduction of the maximal ideal $\m$.
This fact is also obtained as follows.
If $p_g(A)=1$, then $A$ is elliptic (e.g. \cite[p. 425]{wag.ell}).
Therefore, $\ol{\m^3}=\q\ol{\m^2}$ by \thmref{t:ellr}.
Suppose that  $\m=I_Z$ and $\m^2\ne \q\m$. Then $\m$ is not a $p_g$-ideal by \proref{p:sg} (2), namely, $h^1(\cO_X(-Z))=0$. From the exact sequence \eqref{Secondeq} with $n=1$, we have $\ell_A(\ol{\m^2}/\q\m)=1$. Since  $\m^2\ne \q\m$, we obtain $\ol{\m^2}=\m^2$. 
Hence the following ideals coincide:
\[
\q\ol{\m^2}=\q\m^2 \subset \m^3 \subset \ol{\m^3}.
\]
\end{rem}

\begin{lem}\label{l:h1}
Assume that $A$ is an elliptic singularity, $\cO_X(-Z)$ has no fixed component, and $ZE_{min}=0$, where $E_{min}$ is the minimally elliptic cycle.
Let $B$ be the maximal reduced connected cycle such that $ZB=0$ and $\supp (E_{min})\subset B$.
Then $h^1(\cO_X(-Z))=h(B)$ and $n_0(Z) \le 1$.
\end{lem}


\begin{proof}
Let $\{Z_{B_0}, \dots , Z_{B_m}\}$ be the elliptic sequence on $B_0=B$ and let $C=\sum _{i=0}^mZ_{B_i}$. 
By \proref{p:tomari-nef}, $C$ is anti-nef on $B$ and $\chi(C)=0$.
Suppose $E_i\not\subset B$ and $E_i\cap B \ne \emptyset$.
By \proref{p:Eexists} (2), we have that $CE_i\le 1$ and that the cohomological cycle on $Z^{\bot}$ has support in $B$, so $h(B)=h(Z^{\bot})$.
Since $ZE_i<0$ by the definition of $B$, it follows that $Z+C$ is anti-nef on $E$.
By \lemref{l:rohr}, we have $H^1(\cO_X(-Z-C))=0$.
Therefore, by \remref{r:triv-bot}, $h^1(\cO_X(-Z))=h^1(\cO_C(-Z))=h^1(\cO_C)\le h(B)$.
On the other hand, by \lemref{l:nZ} (1) and (3), we have $h^1(\cO_X(-Z))\ge h^1(\cO_X(-n_0Z))=h(B)$.
\end{proof}

\begin{proof}
[Proof of \thmref{t:ellr}]
By \lemref{l:h1}, for any integrally closed $\m$-primary ideal $I$ represented by $Z$, we have $q(I_{nZ})=q(I_Z)$ for $n\ge 1$. By \corref{c:r=n0+1}, we obtain $\bar r (A)\le 2$.
\end{proof}


The invariant $q$ is a function on the set of integrally closed $\m$-primary ideals in $A$. 
So we define a set $\Img_A(q) \subset \Z$ by
\[
\Img_A (q)=\defset{q(I)}{\text{$I\subset A$ is an integrally closed $\m$-primary ideal}}.
\]

By \proref{p:le-cycle}, we have 
\[
\Img_A (q) \subset\{0,1,\dots, p_g(A)\}.
\]

Let $N_0$ denote the set of integers $n_0(W)$, where $W$ runs throught cycles on resolutions $Y$ of $\spec A$ such that $\cO_Y(-W)$ has no fixed component.
Then we define an invariant $n_0(A)$ by $n_0(A)=\sup N_0$.

\begin{prop}\label{p:br=>Imgq}
If $n_0(A)=1$, then $\Img_A (q)=\{0,1,\dots, p_g(A)\}$.
\end{prop}


\begin{proof}
Let $Z>0$ be a cycle on $X$ such that $\cO_X(-Z)$ is generated by global sections and $q(I_Z)=0$ (e.g. \cite[4.5]{OWYgood}).
Take a general element $h\in I_Z$ (see \assref{ass:NC}) and $H:=\di _{\spec A}(h)$.
Let $X_0=X$ and let $\phi_i\: X_i \to X_{i-1}$ be the blowing-up at a point in the intersection of $\supp(C_{X_{i-1}})$ and the strict transform of $H$ on $X_{i-1}$. Let $F_i$ denote the exceptional set of $\phi_i$ and $Z_i:=\phi_i^*Z_{i-1}+F_i$, where $Z_0=Z$. By \proref{p:CC} and \proref{p:CX}, the sequence of blowing-ups $\{\phi_i\}$ ends in a finite number of steps. If $\phi_n$ is the last one, then $Z_n$ is a $p_g$-cycle.
From the exact sequence
\[
0\to \cO_{X_{i}}(-Z_i) \to \cO_{X_{i}}(-\phi_i^*Z_{i-1}) \to \cO_{F_i} \to 0,
\]
we obtain that
\[
0\le h^1(\cO_{X_{i}}(-Z_i)) - h^1(\cO_{X_{i-1}}(-Z_{i-1})) \le 1.
\]
Therefore, there exists a sequence $\{i_0, \dots, i_{p_g(A)}\}\subset \{0,1,\dots, n\}$ such that $h^1(\cO_{X_{i_k}}(-Z_{i_k}))=k$.
By the definition of the cycle $Z_i$, $\cO_{X_{i}}(-Z_{i})$  has no fixed component. Therefore, for each $i$,  $h^1(\cO_{X_{i}}(-nZ_{i}))$ is stable for $n\ge 1$ since $n_0(Z_i)\le 1$.
By \lemref{l:nZ} (4), $\cO_{X_{i_k}}(-2Z_{i_k})$ is generated by global sections and thus $q(I_{2Z_{i_k}})=k$
by the proof of \thmref{t:ellr}.
\end{proof}

\lemref{l:h1} and \proref{p:br=>Imgq} implies the following.
\begin{cor}\label{c:P}
If $A$ is an elliptic singularity, then 
\begin{equation}\label{eq:P=}
\Img_A (q)=\{0,1,\dots, p_g(A)\}.
\end{equation}
\end{cor}


\begin{rem}
Assume that $A$ is an elliptic singularity and $Z>0$ is a $p_g$-cycle.
Let $B$ be the maximal reduced connected cycle such that $ZB=0$ and $\supp (E_{min})\subset B$ and let 
$\{Z_{B_0}, \dots , Z_{B_m}\}$ be the elliptic sequence on $B_0=B$.
Let $Z_{B_{-1}}=Z$ and $D_t=\sum_{i=-1}^t Z_{B_i}$. Then it follows from \lemref{l:h1} that 
$h^1(\cO_X(-D_{i-1}))=h^1(B_i)$ for $0\le i\le m$. 
Therefore $\Img_A(q)=\defset{h^1(B_i)}{i=0,1,\dots, m}\cup \{0\}$.
\end{rem}

The property \eqref{eq:P=} does not imply that $A$ is an elliptic singularity. In fact, we have the following.

\begin{ex}[cf. {\cite[Example 4.6]{OWYgood}}]
\label{ex:highpg} 
Let $C$ be a nonsingular curve of genus $g= 2$ and put
\[
R=\bigoplus_{n\ge 0}H^0(\cO_C(nK_C)).
\]
Suppose that $A$ is the localization of $R$ at $R_+=\bigoplus_{n\ge 1}H^0(\cO_C(nK_C))$ and let $f\: X\to \spec A$ be the minimal resolution.
Then $p_g(A)=3$, $E\cong C$, $\cO_E(-E)\cong \cO_E(K_E)$, $-E^2=2$,  $K_X=-2E=-C_X$, and $\cO_X(-E)$ is generated by global sections.
In particular, $\m=I_E$. It follows that $H^1(\cO_X(-2E))=0$ by the Grauert-Riemenschneider vanishing theorem.

We show that $\Img_A (q)=\{0,1,2,3\}$.
From the exact sequence 
\[
0 \to \cO_X(-E) \to \cO_X \to \cO_E \to 0,
\]
we have $h^1(\cO_X(-E))=p_g(A)-2=1$. Hence $1=q(\m)\in \Img_A (q)$.
Let $h\in \m$ be a general element and suppose $\di _X(h)=E+H_1+H_2$.
Let $\phi\: X'\to X$ be the blowing-up at $E\cap(H_1\cup H_2)$,
 and let $E_i=\phi^{-1}(E\cap H_i)$ and $Z=\phi^*E+E_1+E_2$.
If $E_0$ denote the strict transform of $E$, then $\cO_{E_0}(-Z)\cong \cO_{E_0}$ (cf. \remref{r:triv-bot}), and hence $h^1(\cO_{X'}(-nZ))\ge h^1(\cO_{E_0})=2$ for $n\ge 1$.
Since $C_{X'}=\phi^*(2E)-E_1-E_2$ by \proref{p:CC}, we have $ZC_{X'}=-2$.
By \proref{p:CX}, $h^1(\cO_{X'}(-nZ))\ne 3$. Hence $h^1(\cO_{X'}(-nZ))=2$ for $n\ge 1$.
By \lemref{l:nZ} (4), $\cO_{X'}(-2Z)$ is generated by global sections and $2=q(I_{2Z})\in \Img_A (q)$.
\end{ex}

\begin{prob}
For any normal surface singularity $(A,\m)$, does the equality $\Img_A (q)=\{0,1,\dots, p_g(A)\}$ holds?
\end{prob}

\section{When is the maximal ideal a $p_g$-ideal?}


From \exref{ex:highpg}, we see that in general the maximal ideal is not a $p_g$-ideal.
It is natural to ask for a characterization of normal surface singularities $(A,\m)$ with $q(\m)=p_g(A)$.
In \cite[Example 4.3]{OWYrees}, it is shown that for a complete Gorenstein local ring $A$ with $p_g(A) > 0$, $\m$ is a $p_g$-ideal if and only if $A\cong k[[x,y,z]]/(x^2+g(y,z))$, where $k$ is the residue field of $A$ and $g\in (y,z)^3\setminus (y,z)^4$. 
In this section, we give a geometric characterization of such singularities. So we work on the resolution space.
We assume that $p_g(A)>0$.


Let us recall that for a function $h \in \m$, which has zero of order $a_i$ at
$E_i$, $(h)_E$ denotes a cycle such that $(h)_E=\sum a_iE_i$.

\begin{defn}
The {\em maximal ideal cycle} on $X$ is the minimum of $\defset{(h)_E}{h \in \m}$. 
\end{defn}

A cycle $M>0$ on $X$ is the  maximal ideal cycle if and only if $\cO_X(-M)$ has no fixed component and $\m=H^0(X,\cO_X(-M))$.

\begin{lem}\label{l:paM}
Let $M$ be the maximal ideal cycle on $X$. Then 
$\m$ is a $p_g$-ideal represented by $M$ if and only if $p_a(M)=0$.
\end{lem}
\begin{proof}
From the exact sequence 
$$
0 \to \cO_Y(-M)\to \cO_Y \to \cO_M \to 0,
$$
we have $p_a(M)=p_g(A)-h^1(\cO_X(-M))$. 
Since $\cO_X(-M)$ has no fixed component, the assertion follows from \proref{p:le-cycle}.
\end{proof}

The following theorem is proved by Tomari (see \cite[Corollary 3.12 and Theorem 4.3]{tomari.ell}).
Let us give a proof from our point of view.

\begin{thm}[Tomari]\label{t:mpg}
Let $M$ be the maximal ideal cycle on $X$ and $f'\: X'\to \spec A$ be the blowing-up by $\m$.
Then $p_a(M)=0$ if and only if the following three conditions are satisfied.
\begin{enumerate}
\item $\emb A=\mult A+1$.
\item $X'$ is normal.
\item $\cO_X(-M)$ is generated by global sections.
\end{enumerate}
\end{thm}
\begin{proof}
Assume that $p_a(M)=0$.
By \lemref{l:paM}, $\m$ is a $p_g$-ideal and $\cO_X(-M)$ is generated by global sections.
By \cite[6.2]{OWYgood}, (1) holds.
\proref{p:pgC} implies (2).

Conversely assume that the conditions (1)--(3) are satisfied.
By (1) and Goto--Shimoda \cite[1.1 and 1.4]{Goto-Shimoda}, $G:=\bigoplus_{n\ge 0}\m^n/\m^{n+1}$ is a Cohen-Macaulay ring with $a(G)<0$, where $a(G)$ denote the $a$-invariant of Goto--Watanabe \cite{G-W}.
Then $h^1(\cO_{X'})=0$ by \cite[(1.18)]{tki-w}.
By (2) and (3), $X'$ is obtained by contracting the cycle $M^{\bot}$ on $X$, and there exists the following exact sequence:
\[
0\to H^1(\cO_{X'}) \to H^1(\cO_X) \to H^0(R^1\phi_*\cO_X) \to 0.
\]
This shows that $p_g(A)=\ell_A(R^1\phi_*\cO_X)=h(M^{\bot})$. 
Since $h(M^{\bot})\le h^1(\cO_X(-M))$ by \lemref{l:nZ}, we obtain $h^1(\cO_X(-M))=p_g(A)$.
\end{proof}

\begin{cor}\label{c:Gpg}
If $A$ is Gorenstein and $\m$ is a $p_g$-ideal, then $\mult A=2$.
\end{cor}
\begin{proof}
It follows from \lemref{l:paM} and \thmref{t:mpg} that
 $\emb A=\mult A+1$. Since $A$ is Gorenstein, $\mult A=2$ by \cite[3.1]{sally.tangent}.
\end{proof}


\begin{rem}
If $\m$ is a $p_g$-ideal, then for any general element $h\in \m$, 
$\spec A/(h)$ is a partition curve (see \cite [\S 3] {b-c.rat}), because 
$\delta (A/(h))=\emb A/(h) -1$ by the formula of Morales \cite[2.1.4]{mo.rr}.
Note that if $\m$ is represented on a resolution $X$, the strict transform of $\di_{\spec A}(h)$ on $X$ is nonsingular by \assref{ass:NC}.
\end{rem}


\begin{defn}\label{d:G}
A normal surface singularity $A$ is said to be {\em numerically Gorenstein} if $Z_{K_X}\in \sum _i\Z E_i$. 
The definition is independent of the choice of the resolution.
\end{defn}

It is known that $(A,\m)$ is Gorenstein if and only if 
$(A, \m)$ is numerically Gorenstein
 and $-K_X \sim Z_{K_X}$.

\begin{defn}[Yau {\cite[\S 3]{yau.max}}]
Assume that $A$ is elliptic and numerically Gorenstein.
Let $Z_0\ge \dots \ge  Z_m$ be the elliptic sequence on $E$.
Then  $p_g(A)\le m+1$.
If $p_g(A)=m+1$, $A$ is called a {\em maximally elliptic} singularity.
\end{defn}

\begin{thm}[Yau {\cite[Theorem 3.11]{yau.max}}]\label{t:maxGor}
A maximally elliptic singularity is Gorenstein.
\end{thm}


Let $Z_E$ be the fundamental cycle. 
The number $-Z_E^2>0$ is called the {\em degree} of $A$.
It is known that the degree is independent of the choice of the resolution.


The following result (even more general results) can be recovered from 2.15, 3.10 and 5.10 of \cite{o.numGell} (cf. \cite{nem.ellip}). However we put a proof for readers' convenience.

\begin{lem}
\label{l:maxell}
Assume that $A$ is a numerically Gorenstein elliptic singularity and that $X\to \spec A$ is the minimal resolution.
Moreover, assume that $-Z_E^2=1$. Then we have the following.
\begin{enumerate}
\item Let $E_{min}$ be the minimally elliptic cycle. Then $E$ can be expressed as $E=\supp (E_{min}) \cup \left(\bigcup _{i=0}^{m-1}E_i\right)$ with the following dual graph:
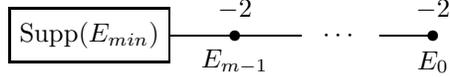
\begin{center}
\begin{picture}(200,25)(60,10)
\put(65,15){{\framebox(60,20){$\supp (E_{min})$}}}
\put(150,25){\circle*{4}}
\put(225,25){\circle*{4}}
\put(125,25){\line(1,0){50}}
\put(205,25){\line(1,0){20}}
\put(150,35){\makebox(0,0){$-2$}}
\put(190,25){\makebox(0,0){$\cdots$}}
\put(225,35){\makebox(0,0){$-2$}}
\put(150,15){\makebox(0,0){$E_{m-1}$}}
\put(225,15){\makebox(0,0){$E_0$}}
\end{picture}
\end{center}
Note that $E_{min}E_{m-1}=1$ by \proref{p:Eexists} (2).

\item $A$ is Gorenstein 
and $Z_E$ coincides with the maximal ideal cycle
if and only if $A$ is a maximally elliptic singularity.
\end{enumerate}
\end{lem}
\begin{proof}
(1) follows from Corollary 2.3 and Table 1 in \cite{yau.hyper}.
We prove (2).

Let $Z_0\ge \dots \ge  Z_m$ be the elliptic sequence on $E$.
Then  $p_g(A)\le m+1$.
It is easy to see that $Z_i=E_{min}+E_{m-1}+\cdots+E_i$.
Let $C'_j:=\sum_{i=j}^mZ_i$. Note that $\cO_{C'_{j+1}}(-Z_{j})=\cO_{C'_{j+1}}(-Z_l)$ for $l\le j$.

Assume that $A$ is Gorenstein 
and $Z_0=Z_E$ is the maximal ideal cycle.
By \remref{r:triv-bot}, we have $\cO_{C'_{j+1}}(-C_{j})\cong \cO_{C'_{j+1}}$ for $0\le j\le m-1$.
It follows from Grauert-Riemenschneider vanishing theorem (or \lemref{l:rohr}) and \cite[Lemma 2.13]{o.numGell} that
$h^1(\cO_X(-Z_0))=h^1(\cO_X(-C_m))+m=m$.
As in the proof of \lemref{l:paM},
we obtain $p_g(A)=h^1(\cO_X(-Z_0))+1=m+1$.

Conversely, assume that $A$ is a maximally elliptic singularity.
Then $A$ is Gorenstein by \thmref{t:maxGor} and $h^1(\cO_X(-Z_0))=m$.
By \proref{p:Eexists}, we easily see that $Z_j$ is $1$-connected (cf. \cite[3.9]{chap}) for $0 \le j \le m$. 
Since $\chi(\cO_{Z_{j+1}}(-C_j))=\chi(Z_{j+1})-C_jZ_{j+1}=0$, we have 
\[
h^1(\cO_{Z_{j+1}}(-C_j))=h^0(\cO_{Z_{j+1}}(-C_j))\le 1
\]
 by \cite[3.11]{chap}.
From the exact sequence
\[
0\to \cO_X(-C_{j+1}) \to \cO_X(-C_j)\to \cO_{Z_{j+1}}(-C_j)\to 0,
\]
we obtain that $0\le h^1(\cO_X(-C_{j}))-h^1(\cO_X(-C_{j+1}))\le 1$ for $0 \le j \le m-1$. 
Thus $h^1(\cO_X(-C_{j}))=h^1(\cO_X(-C_{j+1}))+1$  for $0 \le j \le m-1$. 
Therefore, by \cite[Lemma 2.13]{o.numGell} again, there exists $h\in H^0(\cO_X(-Z_0))$ which maps to the generator of $H^0(\cO_{Z_1}(-Z_0))\cong H^0(\cO_{Z_1})$.
Then the cycles $(h)_E$ and $Z_0$ coincide on $\supp (Z_1)$.
Since $(h)_E$ is anti-nef, we must have $(h)_E=Z_0$.
This shows that $Z_0$ is the maximal ideal cycle.
\end{proof}


\begin{thm}
Assume that $A$ is not a rational singularity, namely, $p_g(A)>0$.
Then the singularity $A$ is Gorenstein and $\m$ is a $p_g$-ideal if and only if $A$ is a maximally elliptic singularity with $-Z_E^2=1$, where $Z_E$ is the fundamental cycle on $E$.
\end{thm}


\begin{proof}
Let $Y\to \spec A$ be the resolution which is obtained by taking the minimal resolution of the blowing-up of $\m$, and let $M$ be the maximal ideal cycle on $Y$.
Let $X_0\to \spec A$ be the minimal resolution and $\phi\:Y\to X_0$ 
the natural morphism.

Assume that $A$ is Gorenstein and $\m$ is a $p_g$-ideal. 
By \corref{c:Gpg}, $\mult A=-M^2=2$.
Since $A$ is Gorenstein, there does not exists a $p_g$-cycle on the minimal resolution $X_0$  by \proref{p:CX}.
Thus $\phi\:Y\to X_0$ is not an isomorphism.
Let $N=\phi_*M$; this is also the maximal ideal cycle on 
$X_0$.
Since $N$ is not a $p_g$-cycle, $\m$ is not represented by $N$, namely,
$\cO_{X_0}(-N)$ is not generated by global sections.
Therefore $-N^2<\mult A=-M^2=2$. 
This implies that $-N^2=1$, and that $\phi$ is the blowing-up at the unique base point of $\cO_{X_0}(-N)$ and $M=\phi^*N+E_0$, where $E_0$ is the exceptional set of $\phi$.
Let $Z_0$ be the fundamental cycle on $X_0$. Since $Z_0 \le N$ and $0<-Z_0^2\le -N^2=1$, we have $Z_0=N$, namely, $N$ is the fundamental cycle.
Since $p_a(M)=(M^2+K_YM)/2+1=0$ by \lemref{l:paM} and $K_YM=(\phi^*K_{X_0}+E_0)(\phi^*N+E_0)=K_{X_0}N-1$, we obtain that $K_{X_0}N=1$.
Thus $p_a(N)=(N^2+K_{X_0}N)/2+1=1$.
Hence $A$ is an elliptic singularity.
By \lemref{l:maxell}, 
 $A$ is a maximally elliptic singularity.

Conversely, assume that $A$ is a maximally elliptic singularity with $-Z_0^2=1$. 
Then $A$ is Gorenstein and $Z_0$ is the maximal ideal cycle by \lemref{l:maxell}.
There exists $h\in H^0(\cO_{X_0}(-Z_0))$ such that $\di_{X_0}(h)=Z_0+H$, where $H$ has no component of $E$.
 Since $-Z_0^2=1$, we have $HZ_0=1$ and that $\cO_{X_0}(-Z_0)$ has just one base point on $\supp (Z_0) \setminus \supp (Z_1)$ which is resolved by the blowing-up at this point (cf. \cite[4.5]{o.numGell}). 
Then $M=\phi^*Z_0+E_0$ and $C_Y=\phi^*(\sum_{i=0}^mZ_i)-E_0$ since $K_{X_0}=-\sum_{i=0}^mZ_i$ (\cite[Theorem 3.7]{yau.max}, \cite[6.8]{tomari.ell}).
Since $Z_0-Z_1$ is reduced (cf. \lemref{l:maxell}), we have $E_0\not \le C_Y$ and thus $\cO_{C_Y}(-M)\cong \cO_{C_Y}$ by \remref{r:triv-bot}. 
Hence $M$ is a $p_g$-cycle by \proref{p:CX}.
\end{proof}

Let us recall that there exist two hypersurface elliptic singularities with $-Z_E^2=1$ which have the same resolution graph, but have different geometric genus.


\begin{ex}[Laufer {\cite[\S V]{la.dbl}}, cf. {\cite[2.23]{nem.ellip}}]
Let $A_1=\C\{x,y,z\}/(x^2+y^3+z^{18})$ and $A_2=\C\{x,y,z\}/(z^2-y(x^4+y^6))$. 
Then the exceptional set $E$ of the minimal resolution $X$ of both these singularities consists of an elliptic curve $E_2$ and $(-2)$-curves $E_0$ and $E_1$, and $E=E_2+E_1+E_0$ is a chain of curves such that $E_2E_1=E_1E_0=1$ (the dual graph of $E$ is similar to that in \lemref{l:maxell}).
We have $p_g(A_1)=3$ and $p_g(A_2)=2$.
So $A_1$ is a maximally elliptic singularity.
For $A_2$, we have that the maximal ideal cycle on $X$ is $M=2E_2+2E_1+E_0$, $\cO_X(-M)$ is generated by global sections since $\mult A_2=2=-M^2$ (cf. \cite[4.6]{chap}), and $h^1(\cO_X(-M))=1=p_g(A_2)-1$ (cf. \lemref{l:h1}). 
\end{ex}

\begin{ex}
By \cite[4.5, 6.3]{o.numGell}, for any positive integer $m$, there exists a numerically Gorenstein elliptic singularity $A$ with elliptic sequence $\{Z_0, \dots, Z_m\}$ on the minimal resolution $X$  such that
 $-Z_0^2=1$, 
\begin{gather*}
C_X=Z_1+ \cdots+Z_m, \quad p_g(A)=m, \quad
M_X=Z_0+Z_1,
\\ \emb A-1=\mult A=-M_X^2+1=3,
\end{gather*}
where $M_X$ denotes the maximal ideal cycle on $X$.
This singularity is {\em not} $\Q$-Gorenstein by \cite[6.1]{o.numGell}.
We claim that $\m$ is a $p_g$-ideal.
The base point of $\cO_X(-M_X)$ is a nonsingular point of $C_X$, which is a point in $\supp (Z_1) \setminus \supp (Z_2)$ by \cite[3.1]{o.numGell}.
Let $\phi\: Y\to X$ be the blowing-up at the base point of $\cO_X(-M_X)$ and $F$ the exceptional set of $\phi$.
Then the maximal ideal cycle $M_Y$ on $Y$ is $\phi^*M_X+F$, and the cohomological cycle on $Y$ is $C_Y=\phi^*C_X-F$.
Since $M_YC_Y=M_XC_X-F^2=Z_1^2-F^2=0$, $M_Y$ is a $p_g$-cycle.
\end{ex}

%

\providecommand{\bysame}{\leavevmode\hbox to3em{\hrulefill}\thinspace}
\providecommand{\MR}{\relax\ifhmode\unskip\space\fi MR }
\providecommand{\MRhref}[2]{%
  \href{http://www.ams.org/mathscinet-getitem?mr=#1}{#2}
}
\providecommand{\href}[2]{#2}

\end{document}